\documentclass[12pt]{article}
\usepackage{graphicx}
\usepackage{graphics}
\usepackage{amsmath}
\usepackage{amsfonts}
\usepackage{amssymb}

\newtheorem{theorem}{Theorem}

\newtheorem{corollary}[theorem]{Corollary}

\newtheorem{definition}[theorem]{Definition}

\newtheorem{lemma}[theorem]{Lemma}

\newtheorem{remark}{Remark}

\newenvironment{proof}[1][Proof]{\textbf{#1.} }{\ \rule{0.5em}{0.5em}}

\def\G{\Gamma}
\def\g{\gamma}

\def\i{\hat{\imath}}

\begin{document}

\title{Rigid gems in dimension n.
\footnote{Work performed under the auspicies of the G.N.S.A.G.A. of
the I.N.D.A.M. (Istituto Nazionale di Alta Matematica "F. Severi" of
Italy) and financially supported by MiUR of Italy (project
``Propriet\`a geometriche delle variet\`a reali e complesse'') and
by the University of Modena and Reggio Emilia (project ``Modelli
discreti di strutture algebriche e geometriche''). }}

\author{Paola Bandieri \ and \ Carlo Gagliardi}

\maketitle

\begin{abstract}
We extend to dimension $n \geq 3$ the concept of $\rho$-pair in a
coloured graph and we prove the existence theorem for minimal rigid
crystallizations of handle-free, closed $n$-manifolds.

\end{abstract}

\section{Introduction}

The concept of \textit{$\rho$-pair} in a $4$-coloured graph was
introduced for the first time by Sostenes Lins in \cite{L}. Roughly
speaking, it consists of two equally coloured edges, which belong to
two or three bicoloured cycles. A graph with no $\rho$-pairs was
then called \textit{rigid} in the same paper, where the following
basic result was proved:

\textit{Every handle-free, closed $3$-manifold admits a rigid
crystallization of minimal order.}

The proof is based on the definition of a particular local move,
called \textit{switching} of a $\rho$-pair. Starting from any gem
$\G$ of a closed, irreducible $3$-manifold $M$, a finite sequence of
such moves, together with the cancelling of a suitable number of
$1$-dipoles, produces a rigid crystallization $\G'$ of the same
manifold $M$, whose order is strictly less than the order of $\G$.

The above existence theorem plays a fundamental r\^ole in the
problem of generating automatically essential catalogues of
$3$-manifolds, with "small" Heegaard genus and/or graph order (see,
e.g., \cite{L}, \cite{BCrG_1}, \cite{CC}, \cite{BCrG_2}, \cite{M}).

In the present paper, we extend the concepts of $\rho$-pair,
switching and rigidity to $(n+1)$-coloured graphs, for $n > 3$.

Our main result is the proof of the existence of a rigid
crystallization of minimal order, for every handle-free
$n$-dimensional, closed manifold. It will be used in a subsequent paper to generate the catalogue of
all $4$-dimensional, closed manifolds, represented by (rigid)
crystallizations of "small" order.

\section{Notations}

In the following all manifolds will be piecewise linear (PL), closed
and, when not otherwise stated, connected. For the basic notions of
PL topology, we refer to \cite{RS} and to \cite{Gl}; "$\cong$" will
mean "\textit{PL-homeomorphic}". For graph theory, see \cite{GT} and
\cite{W}.

We will use the term "graph" instead of  "multigraph". Hence
multiple edges are allowed, but loops are forbidden. As usual,
$V(\G)$ and $E(\G)$ will denote the vertex-set and the edge-set of
the graph $\G$.

If $\G$ is an oriented graph, then each edge ${\bf e}$ is directed
from its first endpoint ${\bf e}(0)$ (also called \textit{tail}) to
its second endpoint ${\bf e}(1)$ (called \textit{head}).

An \textit{$(n+1)$-coloured graph} is a pair $(\G,\gamma)$, where
$\G$ is a graph, regular of degree $n+1$, and $\gamma :
E(\G)\to\Delta_n=\{0,\ldots,n\}$ is a map with the property that, if
${\bf e}$ and ${\bf f}$ are adjacent edges of $E(\G)$, then $\gamma
(\bf e) \neq \gamma (\bf f)$. We shall often write $\G$ instead of
$(\G,\gamma)$.

Let $B$ be a subset of $\Delta_n$. Then, the connected components of
the  graph $\G_B=(V(\G),\gamma^{-1}(B))$ are called
$B$\textit{-residues} of $(\G,\gamma)$. Moreover, for each
$c\in\Delta_n$, we set $\hat{c}=\Delta_n\setminus\{c\}$. If $B$ is a
subset of $\Delta_n$, we define $g_B$ to be the number of
$B$-residues of $\G$; in particular, given any colour $c\in
\Delta_n$, $g_{\hat c}$ denotes the number of components of the
graph $\G_{\hat c}$, obtained by deleting all edges coloured $c$
from $\G$. If $i,j \in \Delta_n, i \neq j$, then $g_{ij}$ denotes
the number of cycles of $\G$, alternatively coloured $i$ and $j$,
i.e. $g_{ij}= g_{\{i,j\}}$.

An isomorphism $\phi : \G\to\G'$ is called a \textit{coloured
isomorphism} between the $(n+1)$-coloured graphs $(\G,\gamma)$ and
$(\G',\gamma')$ if there exists a permutation $\varphi$ of
$\Delta_n$ such that $\varphi\circ\gamma=\gamma'\circ\phi$.

\medskip

A pseudocomplex $K$ of dimension $n$ \cite{HW} with a labelling on
its vertices by $\Delta_n=\{0,\ldots,n\}$, which is injective on the
vertex-set of each simplex of $K$ is called a \textit{coloured
$n$-complex} .

\medskip

It is easy to associate a coloured $n$-complex $K(\G)$ to each
$(n+1)$-coloured graph $\G$, as follows:
\begin{itemize}
\item [-] for each vertex ${\bf v}$ of $\G$, take an $n$-simplex $\sigma ({\bf v})$ and label
its vertices by $\Delta_n$;
\item [-] if ${\bf v}$ and ${\bf w}$ are vertices of $\G$ joined by a $c$-coloured
edge ($c\in\Delta_n$), then identify the $(n-1)$-faces of $\sigma
({\bf v})$ and $\sigma ({\bf w})$ opposite to the vertices labelled
$c$.
\end{itemize}

If $M$ is a manifold of dimension $n$ and $\G$ an $(n+1)$-coloured
graph such that $|K(\G)|\cong M$, then, following Lins \cite{L}, we
say that $\Gamma$ is a \textit{gem}
(\textit{graph-encoded-manifold}) \textit{representing} $M.$

Note that \textit{ $\G$ is a gem of an $n$-manifold $M$ iff, for
every colour $c \in \Delta_n$, each $\hat c$-residue represents
$\mathbb S^{n-1}.$} Moreover, \textit{$M$ is orientable iff $\G$ is
bipartite.}

If, for each $c\in\Delta_n$, $\Gamma_{\hat{c}}$ is connected , then
the corresponding coloured complex $K(\Gamma)$ has exactly $(n+1)$
vertices (one for each colour $c\in\Delta_n$); in this case both
$\G$ and $K(\Gamma)$ are called \textit{contracted}. A contracted
gem $\G$, representing an $n$-manifold $M$, is called a
\textit{crystallization} of $M$.

The \textit{existence theorem} of crystallizations for every
$n$-manifold $M$ was proved by Pezzana \cite{P$_1$}, \cite{P$_2$}.
Surveys on crystallizations theory can be found in \cite{FGG},
\cite{BCaG}.

Let ${\bf x},{\bf y}$ be two  vertices of an $(n+1)$-coloured graph
$\G$ joined by $k$ edges $\{{\bf e}_1,\ldots, {\bf e}_k\}$ with
$\g({\bf e}_h)=i_h$, for $h = 1, \ldots, k$. We call $\Theta=\{{\bf
x},{\bf y}\}$ a \textit{dipole of type k}, \textit{involving
colours} $i_1,\ldots, i_k$, iff ${\bf x}$ and ${\bf y}$ belong to
different $(\Delta_n \setminus \{ i_1,\ldots, i_k\})$-residues of
$\G$.

In this case a new $(n+1)$-coloured graph $\G'$ can be obtained by
deleting ${\bf x},{\bf y}$ and all their incident edges from $\G$
and then joining, for each $i\in\Delta_n \setminus \{i_1,\ldots
,i_k\}$, the vertex $i$-adjacent to ${\bf x}$ to the vertex
$i$-adjacent to ${\bf y}$. $\G'$ is said to be obtained from $\G$ by
\textit{cancelling} (or \textit{deleting}) \textit{the $k$-dipole}
$\Theta$. Conversely $\G$ is said to be obtained from $\G'$ by
\textit{adding the $k$-dipole} $\Theta$.

By a \textit{dipole move}, we mean either the \textit{adding} or the
\textit{cancelling} of a dipole from a gem $\G$.

As proved in \cite{FG}, \textit{two gems $\G$ and $\G'$ represent
PL-homeomorphic manifolds iff they can be obtained from each other
by a finite sequence of dipole moves}.

An $n$-dipole $\Theta=(\bf x, \bf y)$ is often called a
\textit{blob} (see \cite{LM}, where a different calculus for gems is
introduced). If $c$ is the (only) colour not involved in the blob
$\Theta$, and $\bf x', \bf y'$ are the vertices $c$-adjacent to $\bf
x$ and $\bf y$ respectively, then the cancelling of $\Theta$ from
$\G$ produces (in $\G'$) a new $c$-coloured edge $\bf e'$, joining
$\bf x'$ with $\bf y'$. Following Lins, we call the inverse
procedure the \textit{adding of a blob on the edge $\bf e'$}.

Two vertices ${\bf x}, {\bf y}$ of an $(n+1)$-coloured graph $\G$
are called \textit{completely separated} if, for each colour $c \in
\Delta_n$, ${\bf x}$ and ${\bf y}$ belong to two different $\hat
c$-residues. The \textit{fusion graph} $\G \text{fus}({\bf x},{\bf
y})$ is obtained simply by deleting ${\bf x}$ and ${\bf y}$ from
$\G$ and then by gluing together the "hanging edges" with the same
colours.

The following result was first proved, for case (a), in \cite{L}
and, for case (b), in \cite{L} ($n=3$) and in \cite{GV}.

\begin{lemma}\label{1} Let ${\bf x}, {\bf y}$ be two completely separated vertices of
a (possibly disconnected) graph $\G$.
\begin{itemize}
\item[(a)] If ${\bf x}$ and ${\bf y}$ belong to the (only) two different
components $\G^\prime$ and $\G^{\prime\prime}$ of $\G$, representing
two $n$-dimensional manifolds $M^\prime$ and $M^{\prime\prime}$
respectively, then $\G \text{fus}({\bf x},{\bf y})$ is a gem of a
connected sum $M^\prime\# M^{\prime\prime}$.
\item[(b)] If $\G$ is a gem of a (connected) $n$-manifold $M$,
then $\G \text{fus}({\bf x},{\bf y})$ is a gem of $M \# \mathbb H$,
where $\mathbb H$ is either $\mathbb S^{n-1} \times \mathbb S^1$ or
$\mathbb S^{n-1} \tilde \times \mathbb S^1$ (i. e. the orientable or
non-orientable $(n-1)$-sphere bundle over $\mathbb S^1$).
\end{itemize}
\end{lemma}

\smallskip
Note that such a manifold $\mathbb H$ is often called a
\textit{handle} (of dimension $n$). A manifold $M$ is called
\textit{handle-free} if it admits no handles as connected summands
(i.e. if $M$ is not homeomorphic to $M^\prime\# \mathbb H, M^\prime$
being any $n$-manifold).

\smallskip

\section{Switching of $\rho$-pairs.}

Let $(\G,\gamma)$ be an $(n+1)$-coloured graph. Let further $(\bf {e},\bf {f})$ be any pair of edges,
both coloured $c$, of
$\G.$

If we delete ${\bf e},{\bf f}$ from $\G$, we obtain an edge-coloured
graph $\overline \G$, with exactly four vertices of degree $n$
(namely, the endpoints ${\bf u},{\bf v}$ of ${\bf e}$ and the
endpoints ${\bf w},{\bf z}$ of ${\bf f}$).

Now, there are exactly two $(n+1)$-coloured graphs $(\G_1,
\gamma_1), (\G_2, \gamma_2)$ (different from $(\G, \gamma)$) which
can be obtained by adding two new edges (both coloured $c$) to
$\overline \G$ : such edges are either ${\bf e}_1, {\bf f}_1$,
joining ${\bf u}$ with ${\bf w}$ and ${\bf v}$ with ${\bf z}$
respectively, or ${\bf e}_2, {\bf f}_2$, joining ${\bf u}$ with
${\bf z}$ and ${\bf v}$ with ${\bf w}$ respectively. (See Figure 1,
Figure 1a and Figure 1b, where, w.l.o.g., we consider $c=0$)

\bigskip

\smallskip
\centerline{\scalebox{0.5}{\includegraphics{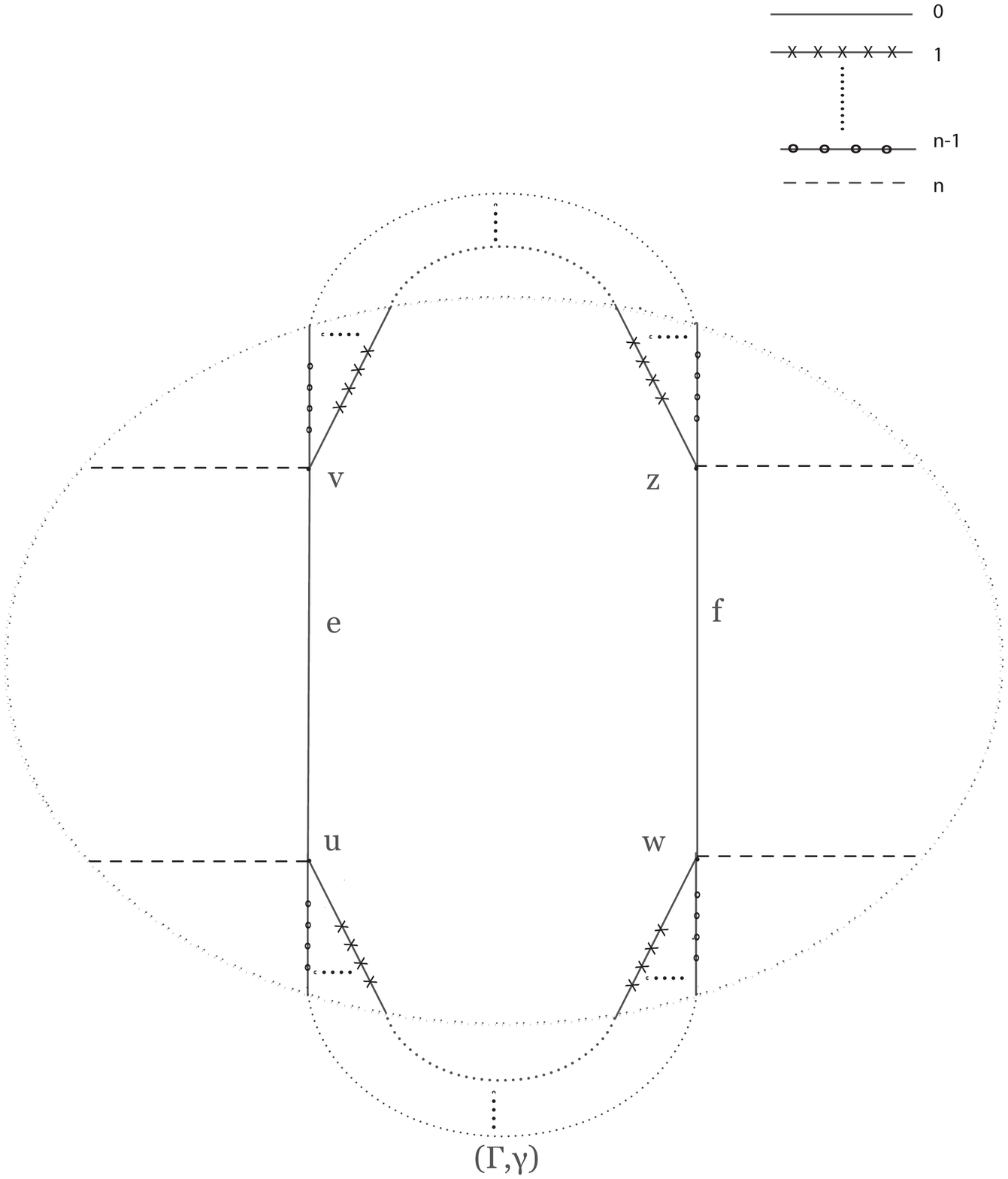}}}

\bigskip
\centerline{\bf Figure 1}

\bigskip

We will say that $(\G_1,\gamma_1)$ and $(\G_2,\gamma_2)$ are
obtained from $(\G,\gamma)$ by a {\it switching} on the pair $({\bf
e},{\bf f}).$

\bigskip

\smallskip
\centerline{\scalebox{0.8}{\includegraphics{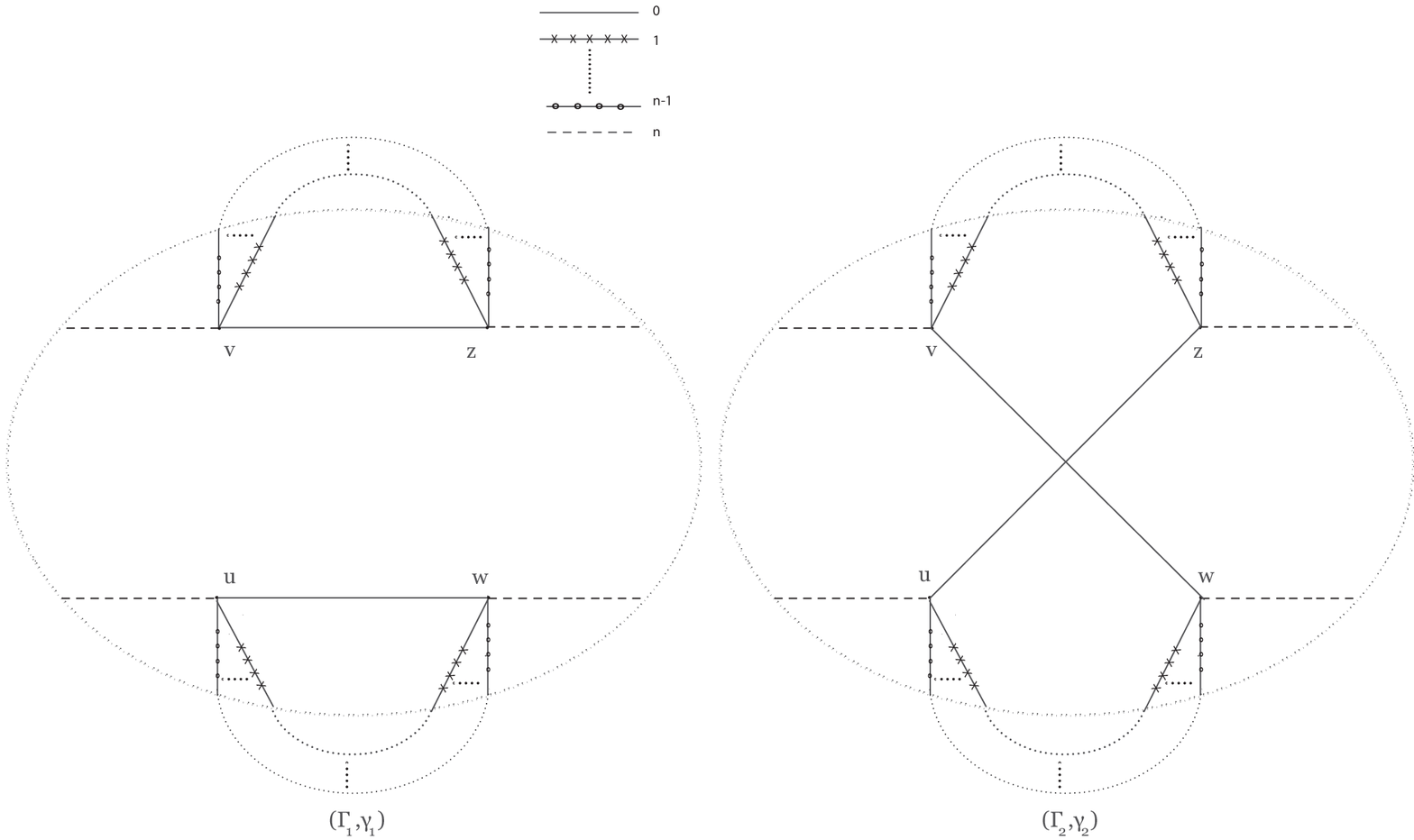}}}
\bigskip
\, \, \, \, \, {\bf Figure 1a \, \, \, \, \, \, \, \, \, \, \, \, \,
\, \, \, \, \, \, \, \, \, \, \, \, \, \, \, Figure 1b}

\bigskip

\medskip

Actually, we are interested in particular pair of equally coloured
edges of $\G$. More precisely:

\medskip

\begin{definition} A pair $R= ({\bf {e}},{\bf {f}})$ of edges of $\G$, with $\gamma({\bf e}) = \gamma({\bf f}) =
c,$ will be called:
\begin{itemize}
\item[(a)] a $\rho_n$-pair involving colour $c$ if, for each
colour $i \in \Delta_n \setminus \{c\}$, we have $\G_{\{c,i\}}({\bf
e}) = \G_{\{c,i\}}({\bf f})$;
\item[(b)] a $\rho_{n-1}$-pair, involving colour $c$, if there
exists a colour $d \neq c$, such that:
\begin{itemize}
\item[($b_1$)] $\G_{\{c,d\}}({\bf e}) \neq \G_{\{c,d\}}({\bf f}),$ and
\item[($b_2$)] for each colour $j \in \Delta_n \setminus \{c,d\}, \G_{\{c,j\}}({\bf e}) = \G_{\{c,j\}}({\bf f}).$
\end{itemize}
\end{itemize}
\end{definition}

The colour $d$ of above will be said to be {\it not involved} in the
$\rho_{n-1}$-pair $R$.

\smallskip

By a {\it $\rho$-pair}, we will mean for short either a
$\rho_n$-pair or a $\rho_{n-1}$-pair.
\medskip

\begin{theorem} \label{n.c.c.} Let $(\G,\gamma)$ be an $(n+1)$-coloured graph, $R =
({\bf e},{\bf f})$ be a $\rho$-pair of $\G$ and let
$(\G_1,\gamma_1)$ be obtained from $(\G,\gamma)$ by any switching of
$R.$ Then:
\begin{itemize}
\item[(a)] if $R$ is a $\rho_{n-1}$- pair, then $\G$ and $\G_1$ have
the same number of components;
\item[(b)] if $R$ is a $\rho_{n}$- pair, then $\G_1$ has at most one more component than $\G$.
\end{itemize}
\end{theorem}
\begin{proof} As before, let us call ${\bf u},{\bf v}$ the endpoints of $\bf e$ and
${\bf w},{\bf z}$ the endpoints of $\bf f$. Let further
$\overline\G$ be the graph obtained by deleting $\bf e$ and $\bf f$
from $\G$.

As it is easy to check, ${\bf u},{\bf v},{\bf w}$ and $\bf z$ lie in
the same component of $\G$.

\medskip

(a) If $R$ is a $\rho_{n-1}$- pair, then $\bf u,\bf v,\bf w$ and
$\bf z$ also lie in the same component of $\overline\G$ (and
therefore of $\G_1$).

For, let $d$ be the colour not involved in $R$. By definition of
$\rho_{n-1}$- pair, $\G_{\{c,d\}}(\bf e)$ and $\G_{\{c,d\}}(\bf f)$
are two different bicoloured cycles of $\G$, the first one
containing $\bf e$ and the second one containing $\bf f$.

Hence there are two paths of $\overline\G$, which join $\bf u$ with
$\bf v$ and $\bf w$ with $\bf z$, respectively.

On the other hand, if $j$ is any colour, $j\neq c,d$, then
$\G_{\{c,j\}}({\bf e}) = \G_{\{c,j\}}({\bf f})$ is a single
bicoloured cycle, containing both $\bf e$ and $\bf f$.

This proves that there is a path of $\overline\G$, which joins $\bf
u$ with either $\bf w$ or $\bf z$.

This completes the proof of (a).

\medskip

(b) If $i\in \Delta_n \setminus \{c\}$, then by definition of
$\rho_{n}$-pair, $\G_{\{c,i\}}({\bf e}) = \G_{\{c,i\}}({\bf f})$.
This proves that there are two paths of $\overline\G$, the first one
joining $\bf u$ with either endpoint of $\bf f$, $\bf w$ say, and
the second one joining $\bf z$ with $\bf v$.

This shows that $\overline\G$ (hence also $\G_1$) has at most one
more component than $\G$.
\end{proof}

\bigskip

 In the following, we will show that in
some particular, but geometrically relevant, cases, it is possible
to choose a "preferred" way to switch a pair of equally coloured
edges of $(\G,\gamma)$.

\medskip

CASE (A):\, \textit{$\G$ bipartite.}

\medskip

If $R = (\bf e,\bf f)$ is any pair of edges, both coloured $c$ (in
particular, if $R$ is a $\rho$-pair) of a bipartite $(n+1)$-coloured
graph $(\G,\gamma)$, then it is easy to see that only one of the two
possible switching of $R$ is again bipartite.

If, further, $V_0, V_1$ are the two bipartition classes of $\G$ and
we orient $\bf e,\bf f$ from $V_0$ to $V_1$, so that their tails
$\bf x_0 = \bf e(0), \bf y_0 = \bf f(0)$ belong to $V_0$, and their
heads $\bf x_1 = \bf e(1), \bf y_1 = \bf f(1)$ belong to $V_1$, the
(bipartite) switching $(\G',\gamma')$ of $R$ is obtained as follows:

\begin{itemize}
\item[(I)] delete $\bf e$ and $\bf f$ from $\G$ (thus obtaining
$\overline\G$);
\item[(II)] join $\bf x_0$ with $\bf y_1$ and $\bf x_1$ with $\bf y_0$ by two new
edges $\bf e', \bf f'$, both coloured $c$.
\end{itemize}

\medskip

CASE (B): \textit{$\G$ non bipartite, with bipartite residues.}

\medskip

If $\G$ is a non bipartite graph, but for each colour $i$,
$\G_{\hat{\imath}}$ has bipartite components (residues), then we
shall consider two subcases

\medskip

Subcase (B$_1$):  \textit{$R = (\bf e,\bf f)$ is a $\rho_{n-1}$-pair
of $\G$, involving colour $c$ and not involving colour $d$.}

\smallskip

Let $\Xi$ be the residue of $\G_{\hat{d}}$ containing both $\bf e$
and $\bf f$ (note that $\bf e$ and $\bf f$ belong to the same
$\hat{\imath}$-residue, because for every colour $i \neq c,d$,
$\G_{\{c,i\}}(\bf e) = \G_{\{c,i\}}(\bf f)$.)

Let $V_0, V_1$ be the two bipartition classes of $\Xi$ (recall that
$\Xi$ is bipartite), As in Case (A), let us orient $\bf e$ from
$V_0$ to $V_1$. Now, the switching of $R = (\bf e,\bf f)$ is the
$(n+1)$-coloured graph $(\G',\gamma')$ , obtained as before (Case
(A)):
\begin{itemize}
\item[(I)] delete $\bf e$ and $\bf f$ from $\G$;
\item[(II)] join $\bf x_0= \bf e(0)$ with $\bf y_1 = \bf f(1)$ and $\bf x_1 = \bf e(1)$ with $\bf y_0 = \bf f(0)$
by two
new edges $\bf e', \bf f'$, both coloured $c$.
\end{itemize}

\medskip

Subcase (B$_2$): \textit{$R = (\bf e,\bf f)$ is a $\rho_{n}$-pair
(involving colour $c$) of $\G$ and $n\geq 3$.}

\smallskip

Let us orient arbitrarily the edge $\bf e$, as before, let us call
$\bf x_0 = \bf e(0)$ and $\bf x_1 = \bf e(1)$. Let now $i$ be any
colour different from $c$. The orientation on $\bf e$ induces a
coherent orientation on all edges of the cycle $\G_{\{c,i\}}(\bf e)$
and, in particular, on the edge $\bf f$ (with the induced
orientation).

Now, we shall prove that the orientation on $\bf f$ (and hence its
tail and its head) is independent from the choice of colour $i$ ($i
\neq c$).

For, let $h$ be any colour of $\Delta_n, h \neq c$, and let $\bf
y_0^h, \bf y_1^h$ be the tail and the head of the edge $\bf f$, with
the orientations induced by the cycle $\G_{\{c,h\}}(\bf e)$ ($\bf e$
being oriented as before).

Let now $j \in \Delta_n, j \neq i,c$. In order to prove that $\bf
y_0^i = \bf y_0^j$ (and, as a consequence $\bf y_1^i = \bf y_1^j$),
let us consider a further colour $k$, with $k \neq i,j,c$.

Note that such a colour $k$ must exist, since $n \geq 3$ and
therefore $\Delta_n$ contains at least four colours.

Let now $\Xi$ be the $\hat{k}$-residue of $\G$, which contains $\bf
e$. $\Xi$ is bipartite and contains both the cycles
$\G_{\{c,i\}}(\bf e)$ and $\G_{\{c,j\}}(\bf e)$. As a consequence,
$\bf y_0^i = \bf y_0^j$. In fact, supposing on the contrary, $\bf
y_0^i = \bf y_1^j$, we could construct an odd cycle of $\Xi$.

The construction of the switching $(\G',\gamma')$ of the
$\rho_{n}$-pair $R = (\bf e,\bf f)$ can now be performed as in the
above cases:
\begin{itemize}
\item[(I)] delete $\bf e$ and $\bf f$ from $(\G,\gamma)$;
\item[(II)] join $\bf x_0$ with $\bf y_1^i$ and $\bf x_1 $ with $\bf y_0^i$ by two new
edges $\bf e', \bf f'$, both coloured $c$.
\end{itemize}

\begin{remark} The above cases include all $\rho$-pairs of gems
representing orientable $n$-manifolds (Case (A)), all
$\rho_{n-1}$-pairs of gems representing non orientable $n$-manifolds
(Case (B$_1$)) and all $\rho_n$-pairs of gems representing non
orientable $n$-manifolds, with $n \geq 3$ (Case (B$_2$)).
\end{remark}

The only remaining case is that of a $\rho_2$-pair of a gem $\G$
representing a non orientable surface, for which it is not always
possible the choice of a standard switching.

In fact, for $n = 2$, the procedure described in Case (B$_2$)
doesn't work, as it depends on the choice of the colour $i$.
\footnote{The case $n=2$ is completely analyzed in \cite{B}, also for graphs representing surfaces with non-empy boundary.}

\section{Main results}

The present section is devoted to prove the following Theorems
\ref{switchrn} and \ref{switchrn-1}, which concern the geometrical
meaning of switching $\rho$-pairs in gems of $n$-dimensional
manifolds.

\bigskip

As in Section 2, let $\mathbb H$ be a handle, i.e. either $(\mathbb
S^{n-1}\times \mathbb S^1)$ or $(\mathbb S^{n-1}\tilde \times
\mathbb S^1).$

\begin{theorem}\label{switchrn} Let $(\G,\g)$ be a gem of a
(connected) $n$-manifold $M$, $n\geqslant 3$, $R=(\bf e, \bf f)$ be
a $\rho_n$-pair in $\G$ and let $(\G', \g')$ be the $(n+1)$-coloured
graph, obtained by switching $R$. Then:
\begin{itemize}
\item[(a)] if $(\G', \g')$ splits into two connected components,
$(\G'_1, \g'_1)$ and $(\G'_2 \g'_2)$ say, then they are gems of two
$n$-manifolds $M'_1$ and $M'_2$ respectively, and $M \cong M'_1 \#
M'_2;$
\item[(b)] if $(\G', \g')$ is connected, then it is a gem of an
$n$-manifold $M'$ such that $M \cong M' \# \mathbb H.$
\end{itemize}
Moreover, if $(\G, \g)$ is a crystallization of $M$, then $(\G',
\g')$ must be connected, and only case (b) may occur.
\end{theorem}

In order to prove Theorem \ref{switchrn}, we need some further
construction and a double sequence of Lemmas, which will be proved
by induction on $n$.

\begin{lemma}\label{1n} -- {\bf step n} \, \, Let $(\Sigma, \sigma)$ be a gem of the $n$-sphere
$\mathbb S^n$, $n\geqslant 2$, $R=(\bf e, \bf f)$ be a $\rho_n$-pair
of $\Sigma$ and let $(\Sigma', \sigma')$ be obtained by switching
$R$. Then $\Sigma'$ splits into two connected components, both
representing $\mathbb S^n.$
\end{lemma}

\medskip

Let now $\G, R=(\bf e, \bf f), \G'$ be as in the statement of
Theorem \ref{switchrn}. Recall that ($n$ being $\geq 3$) any
orientation of $\bf e$ induces a coherent orientation on $\bf f$. As
in Section 2, let ${\bf e(0)}, {\bf f(0)}, {\bf e(1)}$ and ${\bf
f(1)},$ be the ends of $\bf e$ and $\bf f$, so that $\bf e$ is
directed from $\bf e(0)$ to $\bf e(1)$ and $\bf f$ is directed from
$\bf f(0)$ to $\bf f(1).$ Furthermore, after the switching, the new
edges $\bf e', \bf f'$ of $\G'$ join $\bf e(0)$ with $\bf f(1)$ and
$\bf e(1)$ with $\bf f(0)$ respectively. Denote by $\tilde \G$ the
$(n+1)$-coloured graph obtained by adding a blob (i.e. an
$n$-dipole), with vertices $\bf A$ and $\bf B$ on the edge $\bf f'$
of $\G'$ (see Figure 2)

\medskip
\begin{lemma}\label{2n} -- {\bf step n} \, \, With the above notations, if
 $\G$ is a gem of a (connected) $n$-manifold $M$, $n\geq 3$, then:
\begin{itemize}
\item[(i)] $\G'$ (hence also $\tilde \G$) is a gem of a (possibly disconnected)
$n$-manifold $M'$;
\item[(ii)] $\bf e(0)$ and $\bf B$ are two completely separated vertices of $\tilde \G$;
moreover $\G$ coincides with $\tilde \G \text{fus}(\bf e(0),\bf B).$
\end{itemize}
\end{lemma}

\bigskip
\smallskip
\centerline{\scalebox{0.6}{\includegraphics{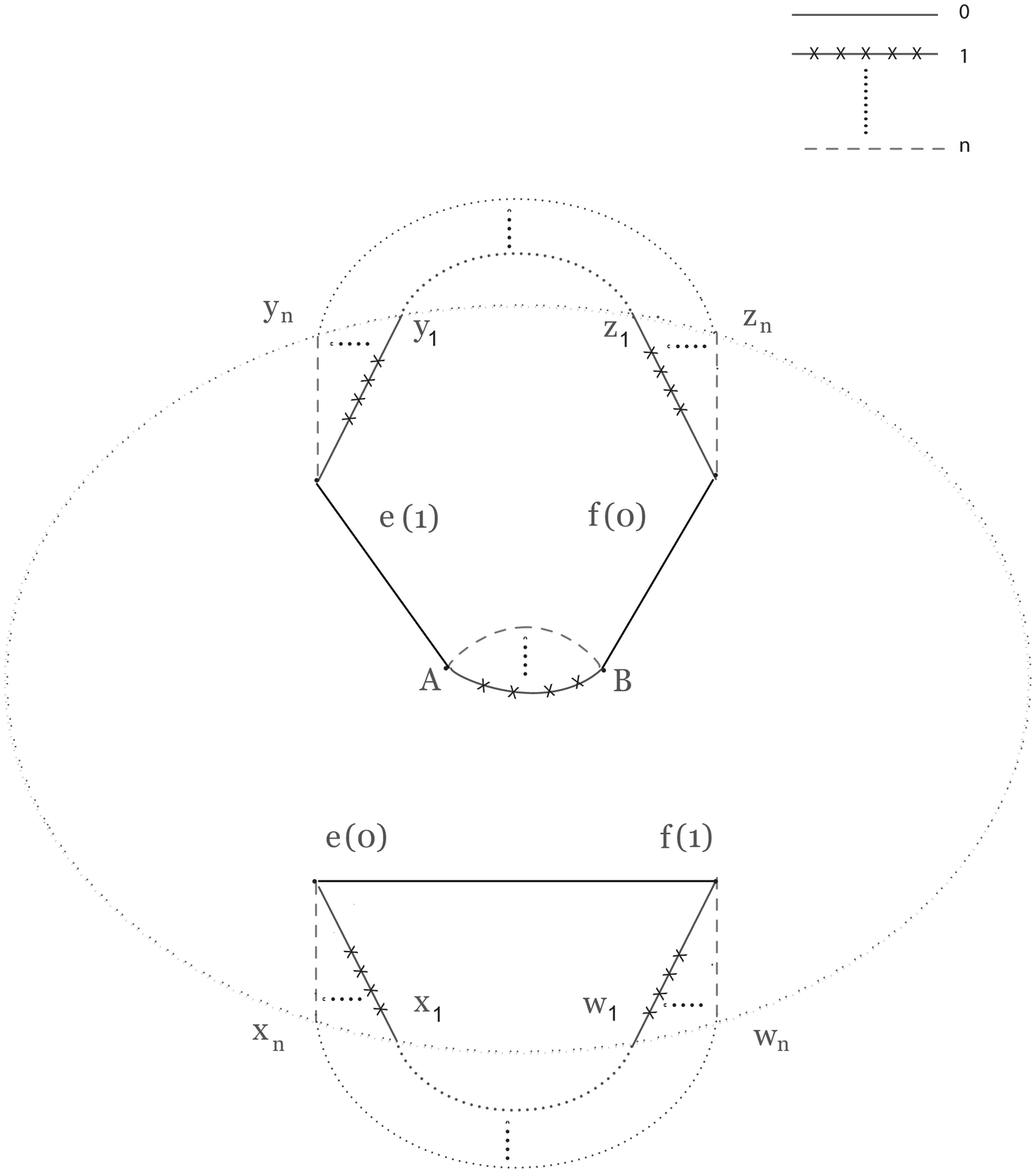}}}
\bigskip
\centerline{\bf Figure 2}

\bigskip

\medskip

\begin{proof} First of all, we repeat here the proof of Lemma \ref{1n}, step 2, which is exactly Corollary 13 of \cite{B}.

Let $(\Sigma, \sigma)$ a $3$-coloured, bipartite graph representing
$\mathbb S^2$. Let $R$ be a $\rho_2$- pair in $\Sigma$ involving
colour $c \in \Delta_2$. Then, by switching $R$ in the only possible
way, we obtain a new graph $(\Sigma', \sigma')$, either connected or
with two connected components. Moreover, if we denote by $d, k$ the
further two colours of $\Delta_2$, then $\Sigma'$ has the same
number of $(d,k)$- coloured cycles ($\hat c$-residues) and one more
$(c,h)$- coloured cycle ($\hat h$-residue), for $h = d,k.$

Hence $\chi(\Sigma') = \chi(\Sigma) + 2 = 4$. This implies that $\Sigma'$ must have two connected components, both representing $\mathbb S^2$.

Now, assuming Lemma \ref{1n}, step $n-1$, we prove Lemma
\ref{2n}, step $n$.

For, let us suppose $\G$ to be a gem of the $n$-manifold $M$. As a
consequence, for each colour $i \in \Delta_n$, all $\i$-residues are
gems of $\mathbb S^{n-1}$. Now, suppose $R=(\bf e, \bf f)$ to be a
$\rho_n$-pair of $\G$, involving color $c$, whose switching produces
the graph $\G'$.

Of course, the switching of $R$ has no effects on the $\hat
c$-residues of $\G$. Hence, each $\hat c$-residue of $\G'$ is
colour-isomorphic to the corresponding one of $\G$, and therefore
represents $\mathbb S^{n-1}$. Let now $i$ be any colour different
from $c$ and let $\Xi$ be the $\i$-residue containing $R$. Of
course, $R$ is a $\rho_{n-1}$-pair of $\Xi$ (where $\Xi$ is a gem of
$\mathbb S^{n-1}$). Hence, by Lemma  \ref{1n}, step $n-1$, the
switching of $R$ splits $\Xi$ into two new $\i$-residues of $\G'$,
both representing $\mathbb S^{n-1}$.

Since all $\i$-residues of $\G$, different from $\Xi$, are left
unaltered by the switching of $R$, $\G'$ is again a gem of a
$n$-manifold $M'$ (with either one or two connected components). Let
now $\tilde \G$ be obtained from $\G'$ by adding a blob (i.e. an
$n$-dipole $\Theta =(\bf A, \bf B)$) on the edge $\bf f'$, with
endpoints $\bf e(0), \bf f(1)$. Of course, $\tilde \G$ is again a
gem of $M'$ and, as it is easy to check, $\tilde \G \text{fus}(\bf
e(0), \bf B)$ is colour-isomorphic to $\G'$ (where the vertex $\bf
A$ plays the role of $\bf e(0)$).

Now, assuming Lemma \ref{2n}, step $n$, we prove Lemma \ref{1n},
step $n$. Let $\Sigma, R=({\bf e},{ \bf f}), \Sigma'$ be as in the
statement of Lemma \ref{1n}. Let further $\tilde \Sigma$ be obtained
by adding a blob $\Theta =(\bf A, \bf B)$ on the edge $\bf f'$ of
$\Sigma'$. Hence, by Lemma \ref{1n}, step $n$, $\Sigma'$ and $\tilde
\Sigma$ are both gems of an $n$-manifold $M'$; moreover ${\bf e}(0)$
and $\bf B$ are completely separated vertices of $\tilde \Sigma$,
and $\Sigma$ is isomorphic to $\tilde \Sigma \text{fus}(\bf e(0),
\bf B)$. If $\Sigma'$ (hence also $\tilde \Sigma$) is connected,
then, by Lemma \ref{1}, the manifold represented by $\Sigma$ must
have a handle $\mathbb H$ as a direct summand, but this is
impossible, since $\Sigma$ represents $\mathbb S^n$, by hypothesis.
Hence $\Sigma'$ (and $\tilde \Sigma$) must split into two components
$\Sigma'_1$, $\Sigma'_2$ say, representing two connected
$n$-manifolds $M'_1, M'_2$ respectively, so that $\mathbb S^n \cong
M'_1 \# M'_2$. But this implies that both $M'_1, M'_2$ are gems of
$\mathbb S^n$, too.

This concludes the proof of Lemmas \ref{1n} and \ref{2n}.
\end{proof}

\smallskip

\begin{proof} (of Theorem \ref{switchrn})
The proof of Theorem \ref{switchrn}, (a) and (b), is now a direct
consequence of Lemma \ref{2n}, Step $n$, and  Lemma \ref{1}.

\smallskip

If, further, $\G_{\hat c}$ is connected, $c$ being the colour
involved in $R$ (in particular, if $\G$ is a crystallization of
$M$), then $\G'$ must be connected, too, and therefore $M \cong M'\#
\mathbb H$.
\end{proof}

\bigskip
\smallskip
\centerline{\scalebox{0.8}{\includegraphics{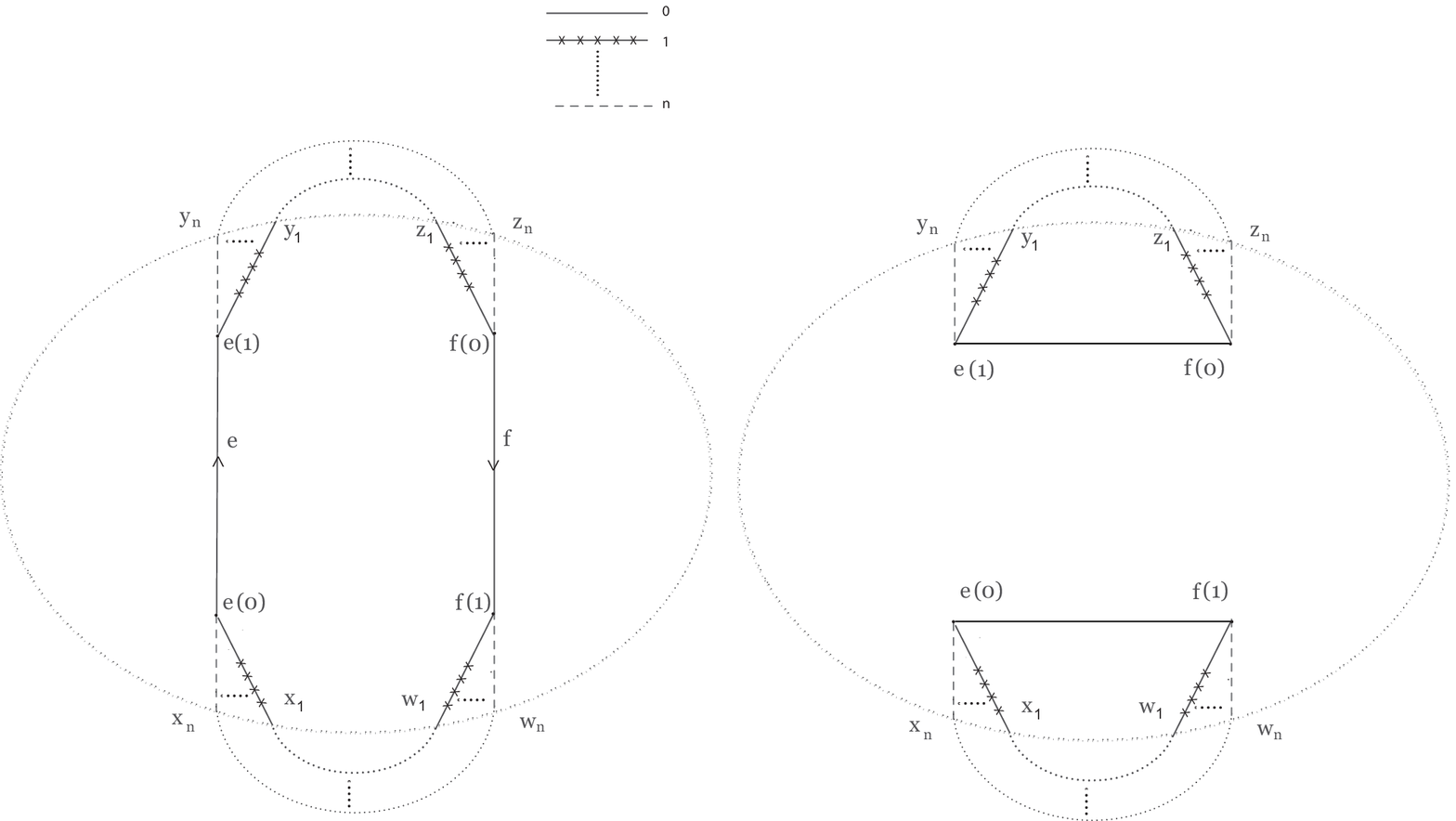}}}
\bigskip
\, \, \, \, \, {\bf Figure 2a \, \, \, \, \, \, \, \, \, \, \, \, \,
\, \, \, \, \, \, \, \, \, \, \, \, \, \, \, Figure 2b}

\bigskip

\smallskip

As a consequence of Theorem \ref{switchrn} and of Corollary 13 of
\cite {B}, we have the following
\begin{corollary}\label{Sn} If $(\Sigma, \sigma)$ is a
crystallization of the $n$-sphere $\mathbb S^n$, $n \geq 2$ then it can't
contain any $\rho_n$-pair.
\end{corollary}

\smallskip

\begin{theorem}\label{switchrn-1} Let $(\G, \g)$ be a gem of a (connected) $n$-manifold $M$,
$R=(\bf e, \bf f)$ be a $\rho_{n-1}$-pair of $\G$ and let $(\G',
\g')$ be obtained by switching $R$. Then $\G'$ is a gem of the same
manifold $M$.
\end{theorem}

\begin{proof} W.l.o.g., let us suppose $c=0$ to be the colour involved and $d=n$ the one not involved in $R$.
By Theorem 2, $\G'$ has the same number of connected components as
$\G$ and, by performing the switching, it is bipartite (resp. non-
bipartite) iff $\G$ is.

Consider the graph $\tilde \G$, obtained by replacing the
neighborhood of $R$ in $\G$ (Figure 3a), with the graph of Figure
3b. The switching of $R$ can be thought as the replacing of the
neighborhood of $R=(\bf e, \bf f)$ by the neighborhood of $R'=(\bf
e', \bf f')$ (see Figure 1a). Consider now the graph $\tilde \G$
obtained by replacing the above neighborhood by the graph of Figure
5, where $\Theta_1$ ($\Theta_2$ resp.) is formed by two vertices
$\bf A', \bf e(0)$ ($\bf B', \bf f(1)$ resp.) joined by $n-1$ edges
coloured $1,\ldots,n-1.$

\bigskip
\smallskip
\centerline{\scalebox{0.5}{\includegraphics{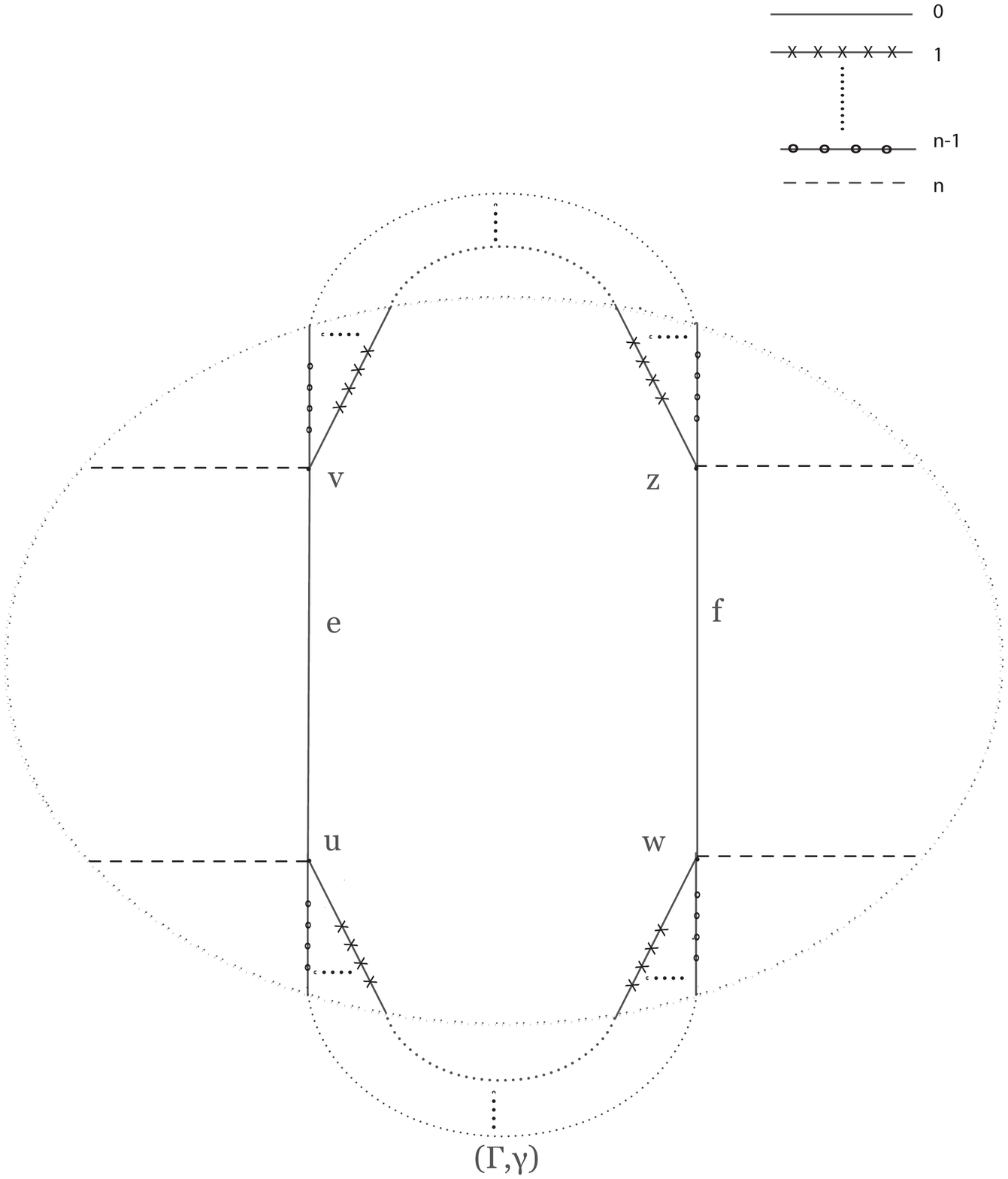}}}
\bigskip
\centerline{\bf Figure 3a}

\bigskip
\smallskip
\centerline{\scalebox{0.5}{\includegraphics{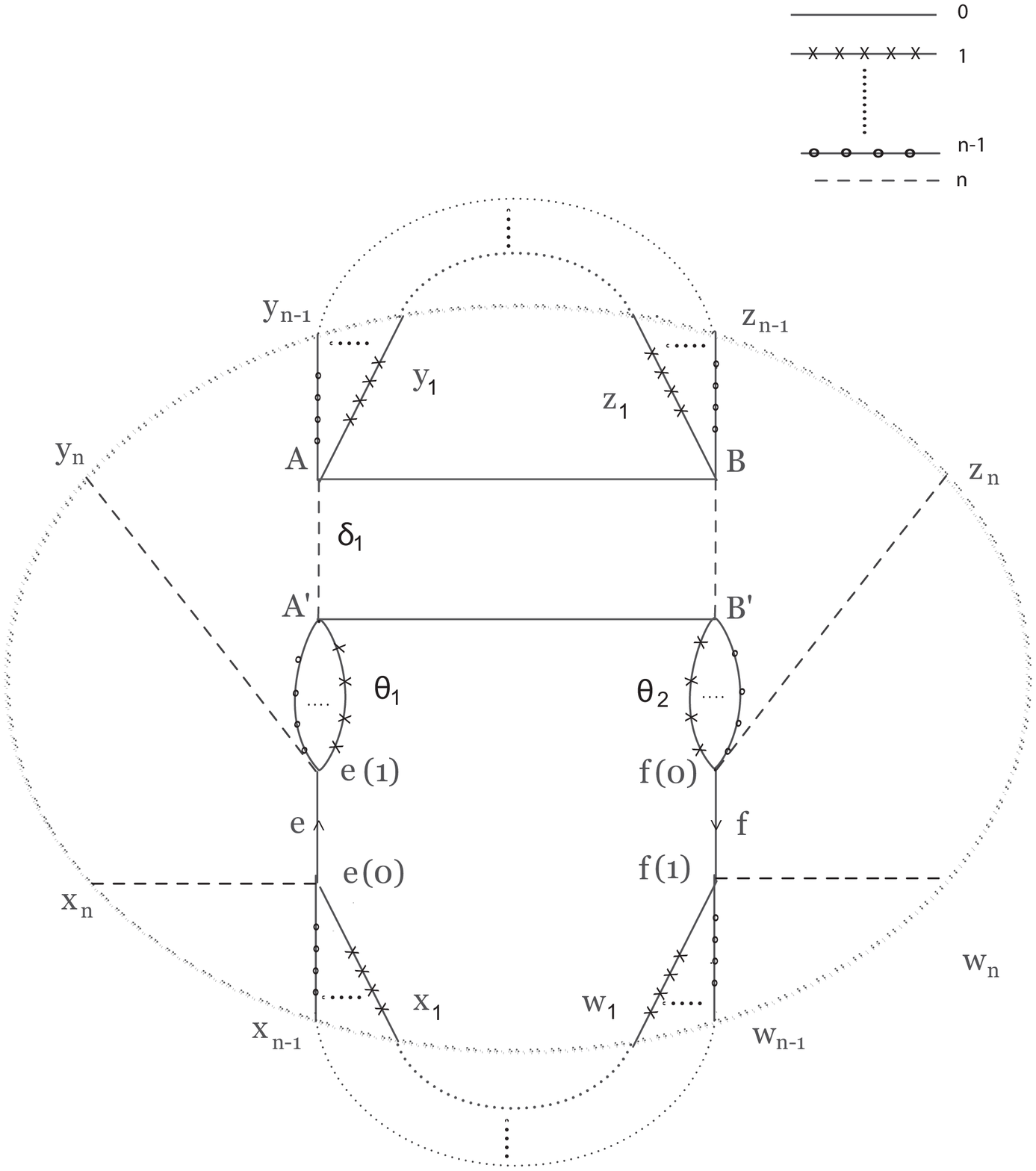}}}
\bigskip
\centerline{\bf Figure 3b}

\bigskip

We will describe two sequences of dipole moves, joining $\tilde \G$
with $\G$ and $\G'$ respectively, thus proving that $\G, \G'$ are
gems of PL-homeomorphic manifolds.

The first sequence starts by considering $\delta_1=(\bf A, \bf A')$,
which is a $1$-dipole. In fact, $\tilde \G_{\hat n}(\bf A') =
\delta_1$, whose further end is $\bf e(1)$; hence the ${\hat
n}$-residue $\tilde \G_{\hat n}(\bf A)$ is different from
$\delta_1.$ By deleting,  the $1$-dipole $\delta_1$ from $\tilde
\G$, yields a $2$-dipole $\delta_2$ with ends $\bf B, \bf B'$, in
fact $\tilde \G_{\hat n}(\bf B')$ is a multiple edge whose further
end is $\bf f(1)$ and which differs from the ${\hat n}$-residue
$\tilde \G_{\hat n}(\bf B)$. By cancelling $\delta_2$, too, we
obtain $\G$ (Figg. 3c and 3d).

$\Theta_1$ and $\Theta_2$ are $(n-1)$-dipoles, since the
$(0,n)$-residue containing $\bf A', \bf B'$ is a quadrilateral cycle
whose vertices are $\bf A, \bf B, \bf A', \bf B'$ only. By deleting
them from $\tilde \G$ (Figg. 3e and 3f), we obtain $\G'$.
\end{proof}

\bigskip

\bigskip
\smallskip
\centerline{\scalebox{0.8}{\includegraphics{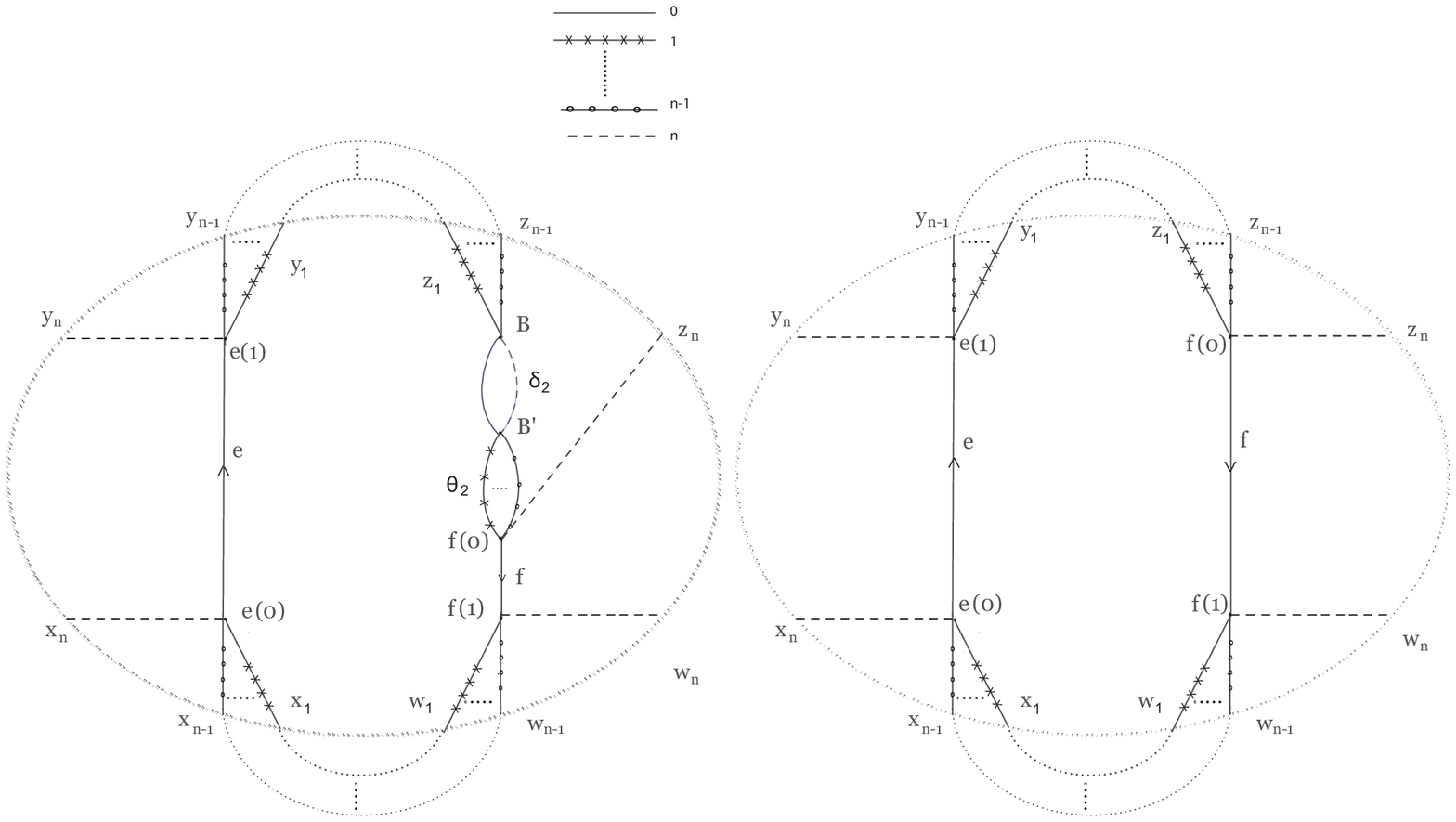}}}
\bigskip
\, \, \, \, \, \, {\bf Figure 3c \, \, \, \, \, \, \, \, \, \, \, \,
\, \, \, \, \, \, \, \, \, \, \, \, \, \, \, \, Figure 3d}

\bigskip

\bigskip
\smallskip
\centerline{\scalebox{0.8}{\includegraphics{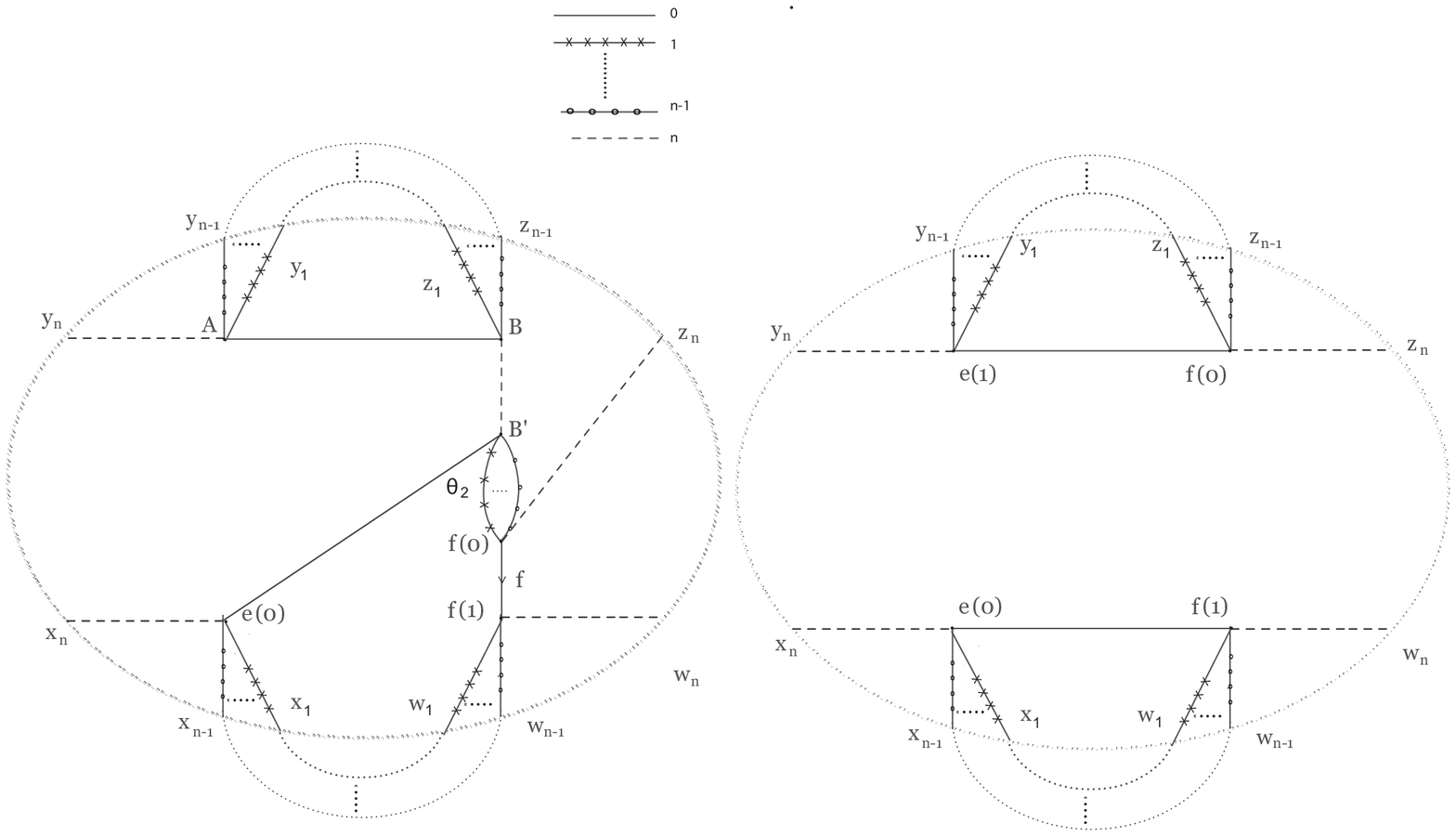}}}
\bigskip
\, \, \, \, \, \, {\bf Figure 3e \, \, \, \, \, \, \, \, \, \, \, \,
\, \, \, \, \, \, \, \, \, \, \, \, \, \, \, \, \, Figure 3f}

\bigskip

\section{Rigid gems}

\bigskip

\begin{definition} An $(n+1)$-coloured graph $(\G,\g)$, $n \geq 3$, is
called \textit{rigid} iff it has no $\rho$-pairs.
\end{definition}

\begin{theorem}\label{rigid} The $(n+1)$-coloured graph $(\G,\g)$, $n \geq 3$, is
rigid iff for each $i\in \Delta_n$, the graph $\G_{\hat{\imath}}$
has no $\rho_{n-1}$-pairs. \footnote{Note that, for $n=2$, the
concept of rigidity has no interest at all. In fact, if $\Gamma$ is
a $3$-coloured graph representing a closed surface, then it contains
$\rho$-pairs: $\rho_2$-pairs, if $\Gamma$ is a crystallization,
either $\rho_1$-pairs or $\rho_2$-pairs, otherwise. Hence, given any
closed surface $M^2$, it can't exist any rigid crystallization of
$M^2$}
\smallskip
\end{theorem}
\begin{proof} Suppose that $(\G,\g)$ is rigid and that there is
a colour $i \in \Delta_n$ such that $\G_{\i}$ has a
$\rho_{n-1}$-pair $R=(\bf e,\bf f)$ of colour $c \in \Delta_n -
\{i\}$. Then $R$ is a $\rho$-pair in $\G$ too, and $\G$ can't be
rigid.

\smallskip

Conversely, If for each $i \in \Delta_n$, $(\G)_{\hat{\imath}}$
contains no $\rho_{n-1}$-pairs, but $(\G,\g)$ isn't rigid, then $\G$
contains at least a $\rho$-pair $R=(\bf e,\bf f)$.

If $R$ is a $\rho_{n}$-pair, then $R$ is a $\rho_{n-1}$-pair in
$(\G)_{\hat{\imath}}$, for each $i \in \Delta_n$.

If $R$ is a $\rho_{n-1}$-pair, and $d$ is the non-involved colour,
then $R$ is a $\rho_{n-1}$-pair in $\G_{\hat{d}}$.
\end{proof}

\begin{theorem}\label{rigcrist} Every closed, connected, handle-free
$n$-manifold $M^n$, $n \geq 3$, admits a rigid crystallization.

Moreover, if $(\G,\g)$ is a crystallization of a closed, connected,
handle-free $n$-manifold $M^n$ of order $p$, then there exists a rigid
crystallization of $M^n$ of order $\le p$.
\end{theorem}

\begin{proof} Starting from any gem of $M^n$ by cancelling a suitable number of $1$-dipoles,
we always can obtain a crystallization of $M^n$ ( see \cite
{FGG}). Suppose now that $\G$ is a crystallization of $M^n$; if $\G$
is rigid, then it is the request crystallization.

If $\G$ has some $\rho_{n-1}$-pair $R=(\bf e, \bf f)$, of colour $c
\in \Delta_n$ and non involving colour $d \in \Delta_n \setminus
\{c\}$, then consider the connected component $\Xi$ of $\G_{\hat d}$
containing both $\bf e$ and $\bf f$. Since $M^n$ is a manifold,
$\Xi$ represents $\mathbb S^{n-1}$ and $R$ is a  $\rho_{n-1}$-pair
in $\G_{\hat d}$, again. For Lemma \ref{1n}, by switching $R$ in
$\G_{\hat d}$, we obtain two connected components, both representing
$\mathbb S^{n-1};$ since $\G_{\hat d}$ is connected (Theorem
\ref{n.c.c.}), then there is at least a $1$-dipole in $\G_{\hat d}$,
whose cancellation reduces the vertex-number.

If $\G$ has some $\rho_{n}$-pair $R=(\bf e, \bf f)$, of colour $c
\in \Delta_n$, then, for each colour $i \in \Delta_n \setminus
\{c\}$, the connected component of $\G_{\i}$ containing $\bf e$ and
$\bf f$, represents $\mathbb S^{n-1}$ and $R$ is a $\rho_{n-1}$-pair
in $\G_{\i}$, as before, by switching $R$ in $\G_{\i}$, we obtain
two connected components, both representing $\mathbb S^{n-1};$ since
$\G_{\i}$ is connected (Theorem  \ref{n.c.c.}), then there is at
least a $1$-dipole in $\G_{\i}$, whose cancellation reduces the
vertex-number, for each $i \in \Delta_n \setminus \{c\}$.
\end{proof}

\medskip

Note that the minimal crystallizations of $\mathbb S^{n-1} \times \mathbb S^1$ and
$\mathbb S^{n-1} \tilde \times \mathbb S^1$ are not rigid (see, e.g., \cite{GV}).
Hence the second statement of Theorem \ref{rigid} is false for handles.

In dimension $3$, there exist rigid crystallizations for $\mathbb
S^{2} \times \mathbb S^1$ and $\mathbb S^{2} \tilde \times \mathbb
S^1$. The minimal ones have order $20$ for $\mathbb S^{2} \times
\mathbb S^1$ and order $14$ for $\mathbb S^{2} \tilde \times \mathbb
S^1$.

For $n>3$, it is easy to construct a rigid crystallization of
$\mathbb S^{n-1} \times \mathbb S^1$, if $n$ is even, and of
$\mathbb S^{n-1} \tilde \times \mathbb S^1$, if $n$ is odd,  both of order $2(2^n-1).$

The remaining cases are still open.

\bigskip
\bigskip

\par \noindent Dipartimento di Matematica Pura ed Applicata,
\par \noindent Universit\`a di Modena e Reggio Emilia,
\par \noindent Via Campi 213 B,
\par \noindent I-41100 MODENA (Italy)
\par \noindent {\it e-mail:  \/} carlo.gagliardi@unimore.it \, \, \, \,  paola.bandieri@unimore.it


\medskip


\begin{thebibliography}{0}

\bibitem{B}
P. Bandieri, {\it $\rho$-pairs in graphs representing surfaces} (to
appear).

\bibitem {BCaG}
P. Bandieri-M.R. Casali-C. Gagliardi, {\it Representing manifolds by
crystallization theory: foundations, improvements and related
results} Atti Sem. Mat. Fis. Univ. Modena. suppl. {\bf 49} (2001),
283--337.

\bibitem{BCrG_1} P. Bandieri-P. Cristofori-C. Gagliardi, {\it Nonorientable
3-manifolds admitting coloured  triangulations with at most 30
tetrahedra}, J. Knot Theory Ramifications {\bf 18}(3) (2009)
381-395.

\bibitem{BCrG_2}  P. Bandieri-P. Cristofori-C. Gagliardi, {\it A census of genus-two 3-manifolds up to
42 coloured tetrahedra}, Discrete Mathematics {\bf 310} (2010) 2469
- 2481.

\bibitem{CC} M.R. Casali-P. Cristofori, {\it A catalogue of orientable 3-manifolds
triangulated by $30$ coloured tethraedra},  J. Knot Theory
Ramifications {\bf 17}(5) (2008), 1 - 23.

\bibitem{FG}  M. Ferri-C. Gagliardi, {\it Crystallization
moves}, Pacific J. Math. {\bf 100} (1982), 85-103.

\bibitem {FGG}
M. Ferri-C. Gagliardi-L. Grasselli, {\it A graph-theoretical
representation of PL-manifolds - A survey on crystallizations} Aeq.
Math. {\bf 31} (1986), 121--141.

\bibitem {Gl}
L. C. Glaser, {\it Geometrical combinatorial topology} Van Nostrand
Reinhold Math. Studies, New York, 1987.

\bibitem{GT}
J.L. Gross-T.W. Tucker, {\it Topological graph theory.} John Wiley
\& Sons, New York, 1987.

\bibitem {GV}
C. Gagliardi-G. Volzone, {\it Handles in graphs and sphere bundles
over $\mathbb S^1$} Europ. J. Combinatorics {\bf 8} (1987),
151--158.

\bibitem{HW}
P. J. Hilton-S. Wylie, {\it An introduction to algebraic
topology-homology theory.} Cambridge Univ. Press, Cambridge, 1960.

\bibitem {L}
S. Lins, {\it Gems, computers and attractors for $3$-manifolds}
Knots and Everything {\bf 5}, 1995.

\bibitem {LM}
S. Lins-M. Mulazzani, {\it Blobs and flips on gems\/}, J. Knot
Theory Ramifications {\bf 15} (2006), no. 8, 1001-1035.

\bibitem {M}
S.V. Matveev, {\it Algorithmic Topolgy and Classification of
3-Manifolds} Springer - Verlag,Berlin and Heidelberg, 2010.


\bibitem{P$_1$}
M.Pezzana, {\it Sulla struttura topologica delle variet\`a
compatte\/}, Atti Sem. Mat. Fis. Univ. Modena, {\bf 23} (1974),
269-277.

\bibitem{P$_2$}
M. Pezzana,  {\it Diagrammi di Heegaard e triangolazione contratta},
 Boll. Un. Mat. Ital. {\bf 12 (4)} (1975), 98--105.

\bibitem {RS}
C. Rourke-B. Sanderson, {\it Introduction to piecewise-linear
topology} Springer Verlag, New York-Heidelberg, 1972.

\bibitem {W}
A. T. White, {\it Graphs, groups and surfaces} North Holland,
revised edition, 1984.


\end{thebibliography}
\end{document}